\documentclass[11pt,hyp,letter]{amsart}
\usepackage{hyperref}
\hypersetup{nesting=true,debug=true,naturalnames=true}
\usepackage{graphicx,amssymb,upref}

\let\<\langle
\let\>\rangle

\let\uml\"

\usepackage{amsmath,amsthm}
\usepackage{amsfonts}
\usepackage{latexsym}
\usepackage{amsbsy}
\usepackage{amssymb,mathscinet}

\numberwithin{equation}{section}

\usepackage{amsfonts,amssymb,amscd,amsmath,enumerate,verbatim,calc,enumerate}
\theoremstyle{plain}
\newtheorem{theorem}{Theorem}[section]
\newtheorem{corollary}[theorem]{Corollary}
\newtheorem{lemma}[theorem]{Lemma}
\newtheorem{proposition}[theorem]{Proposition}

\newtheorem{notation}[theorem]{Notation}
\newtheorem{definition}[theorem]{Definition}

\newtheorem{remark}[theorem]{Remark}
\newtheorem{example}[theorem]{Example}

\newcommand{\bl}[0]{\begin{lemma}}
\newcommand{\el}[0]{\end{lemma}}
\newcommand{\bc}[0]{\begin{corollory}}
\newcommand{\ec}[0]{\end{corollory}}
\newcommand{\bd}[0]{\begin{definition}}
\newcommand{\ed}[0]{\end{definition}}
\newcommand{\bt}[0]{\begin{theorem}}
\newcommand{\et}[0]{\end{theorem}}
\newcommand{\bp}[0]{\begin{proof}}
\newcommand{\ep}[0]{\end{proof}}
\newcommand{\bx}[0]{\begin{example}}
\newcommand{\ex}[0]{\end{example}}
\newcommand{\br}[0]{\begin{remark}}
\newcommand{\er}[0]{\end{remark}}
\newcommand{\bdm}[0]{\begin{displaymath}}
\newcommand{\edm}[0]{\end{displaymath}}
\newcommand{\be}[0]{\begin{equation}}
\newcommand{\ee}[0]{\end{equation}}
\newcommand{\bca}[0]{\begin{caution}}
\newcommand{\eca}[0]{\end{caution}}
\newcommand{\bn}[0]{\begin{note}}
\newcommand{\en}[0]{\end{note}}
\newcommand{\ba}[0]{\begin{axiom}}
\newcommand{\ea}[0]{\end{axiom}}
\newcommand{\bex}[0]{\begin{exercise}}
\newcommand{\eex}[0]{\end{exercise}}
\newcommand{\bpp}[0]{\begin{proposition}}
\newcommand{\epp}[0]{\end{proposition}}
\newcommand{\bnn}[0]{\begin{notation}}
\newcommand{\enn}[0]{\end{notation}}
\newcommand{\bpr}[0]{\begin{property}}
\newcommand{\epr}[0]{\end{property}}
\newcommand{\bo}[0]{\begin{observation}}
\newcommand{\eo}[0]{\end{observation}}

\begin{document}

\title[Invariant subspaces and Curvature]{Covariant representations of subproduct systems: Invariant subspaces and curvature}

\date{\today}

\author[Sarkar]{Jaydeb Sarkar}
\address{ Statistics and Mathematics Unit, Indian Statistical Institute, Bangalore center, 8th Mile, Mysore Road, Bangalore, 560059, India}
\email{jay@isibang.ac.in, jaydeb@gmail.com}

\author[Trivedi]{Harsh Trivedi}
\address{Silver Oak College of Engineering and Technology, Near Bhagwat Vidyapith, Ahmedabad-380061, India.}
\email{harshtrivedi.gn@socet.edu.in, trivediharsh26@gmail.com}

\author[Veerabathiran]{Shankar Veerabathiran}
\address{Ramanujan Institute for Advanced Study in Mathematics,
University of Madras, Chennai (Madras) 600005, India}
\email{shankarunom@gmail.com}


\begin{abstract}
Let $X=(X(n))_{n \in \mathbb{Z_+}}$ be a standard subproduct system
of $C^*$-correspondences over a $C^*$-algebra $\mathcal M.$ Let
$T=(T_n)_{n \in \mathbb{Z_+}}$ be a pure completely contractive,
covariant representation of $X$ on a Hilbert space $\mathcal H$. If
$\mathcal S$ is a closed subspace of $\mathcal H$, then
$\mathcal{S}$ is invariant for $T$ if and only if there exist a
Hilbert space $\mathcal{D}$, a representation $\pi$ of $\mathcal M$
on $\mathcal D,$ and a partial isometry $\Pi:
\mathcal{F}_X\bigotimes_{\pi}\mathcal{D}\to \mathcal{H}$ such that
\[
\Pi (S_n(\zeta)\otimes I_{\mathcal{D}})=T_n(\zeta)\Pi \quad \quad
(\zeta\in X(n), n \in \mathbb{Z_+}),
\]
and $\mathcal S = \mbox{ran~} \Pi$, or equivalently, $P_{\mathcal
S}=\Pi\Pi^*.$ This result leads us to a list of consequences
including Beurling-Lax-Halmos type theorem and other general observations on
wandering subspaces. We extend the notion of curvature for
completely contractive, covariant  representations and analyze it in
terms of the above results.
\end{abstract}

\maketitle
\tableofcontents

\section{Introduction}

Initiated by Gelu Popescu in \cite{MR1670202}, noncommutative
Poisson transforms, and subsequently the explicit and analytic
construction of isometric dilations, have been proved to be an
extremely powerful tools in studying the structure of commuting and
noncommuting tuples of bounded linear operators on Hilbert spaces.
This is also important in noncommutative domains (and subsequently,
noncommutative varieties) classification problems in the operator
algebras (see \cite{MR2643314}, \cite{GP10}, \cite{GP11} and
references therein).

In \cite{MR1668582} Arveson used similar techniques to generalize
Sz.-Nagy and Foias dilation theory for commuting tuple of row
contractions. These techniques have also led to further recent
development \cite{H16, MR1648483, MR2005882, MR2186099, MR2481905,
MR1822685, MR2608451, MR2793449} on the structure of bounded linear
operators in more general settings. In particular, in
\cite{MR2481905} Muhly and Solel introduced Poisson kernel for
completely contractive, covariant representations over
$W^*$-correspondences. The notion of Poisson kernel for completely
contractive, covariant representations over a subproduct system of
$W^*$-correspondences was introduced and studied by Shalit and Solel
in \cite{MR2608451}. This approach was further investigated by
Viselter \cite{MR2793449} for the extension problem of completely contractive, covariant
representations of subproduct systems to $C^*$-representations of
Toeplitz algebras.

Covariant representations on subproduct systems are important since
it is one of the refined theories in operator theory and operator
algebras that provides a unified approach to study commuting as well
as noncommuting tuples of operators on Hilbert spaces.

The main purpose of this paper is to investigate a
Beurling-Lax-Halmos type invariant subspace theorem in the sense of
\cite{MR3346131,J2}, and the notion of curvature in the sense of
Arveson \cite{A2000}, Popescu \cite{MR1822685} and Muhly and Solel
\cite{MR2005882} for completely contractive, covariant representations of
standard subproduct systems.

The plan of the paper is the following. In Section 2, we recall
several basic results from \cite{MR2793449} including the
intertwining property of Poisson kernels. In Section 3 we obtain an
invariant subspace theorem for pure completely contractive, covariant
representations of standard subproduct systems. As an immediate
application we derive a Beurling-Lax-Halmos type theorem. Our
objective in Section 4 is to extend, several results on curvature of
a contractive tuple by Popescu \cite{MR1822685,MR2643314}, for
completely contractive, covariant representations of subproduct
systems. We first define the curvature for completely contractive,
covariant representations of standard subproduct systems. This
approach is based on the definition of curvature for a completely
contractive, covariant representation over a $W^*$-correspondence due
to Muhly and Solel \cite{MR2005882}. The final section is composed of
several results on wandering subspaces which are motivated from our
invariant subspace theorem. This section generalizes \cite[Section
5]{BEDS} on wandering subspaces for commuting tuple of bounded
operators on Hilbert spaces.

\section{Notations and prerequisites}

In this section, we recall some definitions and properties about
$C^*$-correspondences and subproduct systems (see \cite{MR0355613},
\cite{La95}, \cite{MR1648483}, \cite{MR2608451}).

Let $\mathcal M$ be a $C^*$-algebra and let $E$ be a Hilbert
$\mathcal M$-module. Let $\mathcal L(E)$ be the $C^*$-algebra of all
adjointable operators on $E$. The module $E$ is said to be a {\it
$C^*$-correspondence over $\mathcal M$} if it has a left $\mathcal
M$-module structure induced by a non-zero $*$-homomorphism
$\phi:\mathcal M\to \mathcal L(E)$ in the following sense
\[
a\xi:=\phi(a)\xi \quad \quad (a\in\mathcal M, \xi\in E).
\]
All such $*$-homomorphisms considered in this article are
non-degenerate, which means, the closed linear span of
$\phi(\mathcal M)E$ equals $E.$ If $F$ is another
$C^*$-correspondence over $\mathcal M,$ then we get the notion of
tensor product $F\bigotimes_{\phi} E$ (cf. \cite{La95}) which
satisfy the following properties:
\[
(\zeta_1 a)\otimes \xi_1=\zeta_1\otimes \phi(a)\xi_1,
\]
\[
\langle\zeta_1\otimes\xi_1,\zeta_2\otimes\xi_2\rangle=\langle\xi_1,\phi(\langle\zeta_1,\zeta_2\rangle)\xi_2\rangle
\]
for all $\zeta_1,\zeta_2\in F;$ $\xi_1,\xi_2\in E$ and $a\in\mathcal
M.$

Assume $\mathcal M$ to be a $W^*$-algebra and $E$ is a Hilbert
$\mathcal M$-module. If $E$ is self-dual, then $E$ is called a {\it
Hilbert $W^*$-module} over $\mathcal M.$ In this case, $\mathcal
L(E)$ becomes a $W^*$-algebra (cf. \cite{MR0355613}). A
$C^*$-correspondence over $\mathcal M$ is called a {\it
$W^*$-correspondence} if $E$ is self-dual, and if the
$*$-homomorphism $\phi:\mathcal M\to \mathcal L(E)$ is normal. When
$E$ and $F$ are $W^*$-correspondences, then their tensor product
$F\bigotimes_{\phi} E$ is the self-dual extension of the above
tensor product construction.
\begin{definition}
Let  $\mathcal M$ be a $C^*$-algebra, $\mathcal H$ be a Hilbert
space, and $E$ be a $C^*$-correspondence over $\mathcal M$. Assume
$\sigma:\mathcal M\to B(\mathcal H)$ to be a representation and $T:
E\to B(\mathcal H)$ to be a linear map. The tuple $(T,\sigma)$ is
called a {\rm covariant representation} of $E$ on $\mathcal H$ if
\[
T(a\xi a')=\sigma(a)T(\xi)\sigma(a') \quad \quad (\xi\in E,
a,a'\in\mathcal M).
\]
In the $W^*$-set up, we additionally assume that $\sigma$ is normal
and that $T$ is continuous with respect to the $\sigma$-topology of
$E$ (cf. \cite{MR945550}) and ultra weak topology on $B(\mathcal
H)$. The covariant representation is called {\rm completely
contractive} if $T$ is completely contractive. The covariant
representation $(T,\sigma)$ is called {\rm isometric} if
\[
T(\xi)^*T(\zeta)=\sigma(\langle \xi,\zeta\rangle) \quad \quad
(\xi,\zeta\in E).
\]
\end{definition}
The following important lemma is due to Muhly and Solel \cite[Lemma
3.5]{MR1648483}:
\begin{lemma}
The map $(T,\sigma)\mapsto \widetilde T$ provides a bijection
between the collection of all completely contractive, covariant
representations $(T,\sigma)$ of $E$ on $\mathcal H$ and the
collection of all contractive linear maps
$\widetilde{T}:~\mbox{$E\bigotimes_{\sigma} \mathcal H\to \mathcal
H$}$ defined by
\[
\widetilde{T}(\xi\otimes h):=T(\xi)h \quad \quad (\xi\in E,
h\in\mathcal H),
\]
and such that $\widetilde{T}(\phi(a)\otimes I_{\mathcal
H})=\sigma(a)\widetilde{T}$, $a\in\mathcal M$. Moreover, $\widetilde
T$ is isometry if and only if $(T,\sigma)$ is isometric.
\end{lemma}

\begin{example}\label{eg-1}
Assume $E$ to be a Hilbert space with an orthonormal basis
$\{e_i\}^n_{i=1}.$ Any contractive tuple $(T_1, \ldots,T_n)$ on a
Hilbert space $\mathcal H$ can be realized as a completely
contractive, covariant representation $(T,\sigma)$ of $E$ on
$\mathcal H$ where $T(e_i):=T_i$ for each $1\leq i\leq n,$ and when
the representation $\sigma$ maps every complex number $\lambda$ to
the multiplication operator by $\lambda$.
\end{example}

Now we recall several definitions and results from \cite{MR2793449}
which are essential for our objective. We will use $A^*$-algebra, to
denote either $C^*$-algebra or $W^*$-algebra, to avoid repetitions
in statements. Similarly we also use $A^*$-module and
$A^*$-correspondence.
\begin{definition}\label{def2}
Let $\mathcal{M}$ to be an $A^*$-algebra and $X=(X(n))_{n \in
\mathbb{Z_+}}$ be a sequence of $A^*$-correspondences over
$\mathcal{M}$. Then $X$ is said to be a {\rm subproduct system} over
$\mathcal{M}$ if $X(0)=\mathcal{M},$ and for each $n,m \in
\mathbb{Z_+}$ there exist a coisometric, adjointable bimodule
function
$$U_{n,m}:X(n)\mbox{$\bigotimes$} X(m) \to X(n+m),$$ such that
\begin{itemize}
\item[(a)] the maps $U_{n,0}$ and $U_{0,n}$ are the right and the left actions of $\mathcal{M}$ on $X(n)$,
respectively, that is,
\[
U_{n,0}(\zeta \otimes a):=\zeta a,~U_{0,n}(a \otimes \zeta):=a\zeta
\quad \quad (\zeta\in X(n),~a \in \mathcal M, n\in\mathbb{Z_+}),
\]
\item [(b)] the following associativity property holds for all $n,m,l \in \mathbb{Z_+}$;
\[
U_{n+m,l}(U_{n,m}\otimes I_{X(l)})=U_{n,m+l}(I_{X(n)} \otimes
U_{m,l}).
\]
If each coisometric maps are unitaries, then we say the family $X$
is a {\rm product system}.
\end{itemize}
\end{definition}
\begin{definition} \label{def1}
Let $\mathcal{M}$ be an $A^*$-algebra and let $ X=(X(n))_{n \in
\mathbb{Z_+} }$ be a subproduct system over $\mathcal{M}$. Assume
$T=(T_n)_{n \in \mathbb{Z_+}}$ to be a family of linear
transformations $T_n: X(n) \to B(\mathcal{H})$, and define
$\sigma:=T_0$. Then the family $T$ is called a {\rm completely
contractive, covariant representation of $X$ on $\mathcal{H}$} if
\begin{itemize}
\item
[(i)] for every $n \in \mathbb{Z_+}$, the pair $(T_n, \sigma)$ is a completely contractive, covariant representation of the $A^*$-correspondence $X(n)$ on $\mathcal{H}$, and \\
\item
[(ii)] for every $n,m \in \mathbb{Z_+},~ \zeta \in X(n)$ and $ \eta \in X(m)$,
\be\label{eqn1} T_{n+m}(U_{n,m}(\zeta \otimes \eta))=T_n(\zeta)T_m(\eta). \ee
\end{itemize}
\end{definition}
For $n \in \mathbb{Z_+}$ define the contractive linear map
$\widetilde{T}_n:X(n) \bigotimes_{\sigma} \mathcal{H} \to
\mathcal{H}$ as (see \cite{MR1648483})
\begin{equation}\label{eqn3}
\widetilde{T}_n(\zeta \otimes h):=T_n(\zeta)h \quad \quad (\zeta \in
X(n),~h \in \mathcal{H}).
\end{equation}
Thus we can replace (\ref{eqn1}) by
\[
\widetilde{T}_{n+m}(U_{n,m} \otimes
I_{\mathcal{H}})=\widetilde{T}_n(I_{X(n)} \otimes \widetilde{T}_m).
\]

\begin{example}\label{eg-2}
The Fock space $\mathcal{F}_X:=\bigoplus_{n \in \mathbb{Z_+}}X(n)$
of a subproduct system $X=(X(n))_{n \in \mathbb{Z_+}}$  is an
$A^*$-correspondence over $\mathcal{M}$. For each $n \in
\mathbb{Z_+}$, we define a linear map $S^X_n: X(n) \to
\mathcal{L}(\mathcal F_X) $ by
$$S^X_n(\zeta)\eta:=U_{n,m}(\zeta \otimes \eta)$$ for every $m \in
\mathbb{Z_+},$  $\zeta\in X(n)$ and $\eta \in X(m)$. When $ n \neq
0$ we call each operator $S^X_n$ a {\it creation operator} of
$\mathcal{F}_X$, and the family $S^X:=(S^X_n)_{n \in  \mathbb{Z_+}}$
is called an {\it $X$-shift}. It is easy to verify that the family
$S^X$ is indeed a completely contractive, covariant representation
of $\mathcal{F}_X$. From the Definition \ref{def2} it is easy to see
that, for each $a\in \mathcal{M},$ the map  $S^X_0(a)= \phi
_{\infty}(a):\mathcal{F}_X\to \mathcal{F}_X$ maps
$(b,\zeta_1,\zeta_2,\ldots)\mapsto (a b, a\zeta_1,a
\zeta_2,\ldots)$.
\end{example}

Let $\mathcal M$ to be an $A^*$-algebra, and let $X=(X(n))_{n \in
\mathbb{Z_+}}$ be an $A^*$-correspondences over $\mathcal{M}$. Then
$X$ is said to be a \textit{standard subproduct system} if
$X(0)=\mathcal{M},$ and for any $n,m \in \mathbb{Z_+}$ the bimodule
$X(n+m)$ is an orthogonally complementable sub-module of
$X(n)\bigotimes X(m)$.

Let $X=(X(n))_{n \in \mathbb{Z_+}}$ be a standard subproduct system
and $E:=X(1)$. Then for each $n,$ the bi-module $X(n)$ is an
orthogonally complementable sub-module of $E^{\otimes n}$ (here
$E^{\otimes 0}=\mathcal{M}$), and hence there exists an  orthogonal
projection $p_n \in \mathcal{L}(E^{\otimes n})$ of $E^{\otimes n}$
onto $X(n)$. We denote the orthogonal projection $\bigoplus_{n \in
\mathbb{Z_+}}p_n$ of $\mathcal{F}_E$, the Fock space of the product
system $E=(E^{\otimes n})_{n \in \mathbb{Z_+}}$ with trivial
unitaries, onto $\mathcal{F}_X$ by $P$.

Note also that here the projections $(p_n)_{n \in \mathbb{Z_+}}$ are
bimodule maps and
\[
p_{n+m}=p_{n+m}(I_{E^{\otimes n}} \otimes p_{m})=p_{n+m}(p_{n}
\otimes I_ {E^{\otimes m}}),
\]
for all $n,m \in \mathbb{Z_+}$. This implies that if we define each
$U_{n,m}$ to be the projection ${p_{n+m}}$ restricted to
${X(n)\bigotimes X(m)},$ then every standard subproduct system
becomes a subproduct system over $\mathcal{M}.$ In this case
(\ref{eqn1}) reduces to
\[
T_{n+m}(p_{n+m}(\zeta \otimes \eta))=T_n(\zeta)T_m(\eta)~\mbox{for
all}~\zeta\in E^{\otimes n}\mbox{ and }\eta\in E^{\otimes m},
\]
and (\ref{eqn3}) becomes \be \label{eqn6}
\widetilde{T}_{n+m}(p_{n+m} \otimes I_{\mathcal{H}})|_{X(n)
\bigotimes X(m) \bigotimes_{\sigma}
\mathcal{H}}=\widetilde{T}_n(I_{X(n)} \otimes \widetilde{T}_m). \ee
Taking adjoints on both the sides we obtain \be \label{eqn5}
{\widetilde{T}}_{n+m}^*=(I_{X(n)} \otimes
{\widetilde{T}_m}^*){\widetilde{T}_n}^* \quad \quad (n,m\in
\mathbb{Z_+}).\ee Note that for the sake of convenience we ignored
the embedding of $X(n+m) \bigotimes_{\sigma} \mathcal{H}$ into
$X(n)\bigotimes X(m) \bigotimes_{\sigma} \mathcal{H}$ in the
previous formula. We further deduce that \be ~\label{eqn4}
{\widetilde{T}_{n+1}}^*=(I_E \otimes
{\widetilde{T}_n}^*){\widetilde{T}_1}^*=(I_{X(n)} \otimes
{\widetilde{T}_1}^*){\widetilde{T}_n}^*,\ee and
\[
{\widetilde{T}_n}^*=(I_{X(n-1)} \otimes
{\widetilde{T}_1}^*)(I_{X(n-2)} \otimes
{\widetilde{T}_1}^*)\ldots(I_E \otimes
{\widetilde{T}_1}^*){\widetilde{T}_1}^*,
\]
for all $n \in \mathbb{Z_+}$.

\begin{example}\label{eg-3}
If $X(n)$ is the $n$-fold symmetric tensor product of the Hilbert
space $X(1),$ then $X=(X(n))_{n\in\mathbb Z_+}$ becomes a standard
subproduct system of Hilbert spaces (cf. \cite[Example
1.3]{MR2608451}). Moreover, let $\{e_1,\ldots,e_d\}$ be an
orthonormal basis of $X(1)$. Then
\[
T \leftrightarrow (T_1(e_1),T_1(e_2),\ldots,T_1(e_d))
\]
induces a bijection between the set of all completely contractive
covariant representations $T$ of $X$ on a Hilbert space $\mathcal H$
onto the collection of all commuting row contractions $(T_1,
\ldots,T_d)$  on $\mathcal H$ (cf. \cite[Example 5.6]{MR2608451}).
\end{example}

\vspace{0.2in}

Before proceeding to the notion of Poisson kernels, we make a few
comments:

\begin{enumerate}

\item  We use the symbol sot-$\lim$ for the limit
with respect to the strong operator topology. From Equation
\ref{eqn4} we infer that $\{\widetilde{T}_n {\widetilde{T}_n^*}\}_{n
\in \mathbb{Z_+}}$ is a decreasing sequence of positive
contractions, and thus $Q:=\mbox{sot-}\displaystyle\lim_{n
\rightarrow \infty}\widetilde{T}_n {\widetilde{T}_n^*}$ exists. If
$Q=0$, then we say that the covariant representation $T$ is {\it
pure}. Note that $T$ is pure if and only if
$\mbox{sot-}\displaystyle\lim_{n \rightarrow
\infty}{\widetilde{T}_n^*}=0$.

\item Let $\psi$ be a representation of $\mathcal M$ on a Hilbert
space $\mathcal E$. Then the induced covariant representation $S
\otimes I_{\mathcal E}:=(S_n(\cdot)\otimes I_{\mathcal
E})_{n\in\mathbb{Z}_+}$ is pure, where each $S_n(\cdot)\otimes
I_{\mathcal E}$ is an operator from $X(n)$ into $B(\mathcal F_X
\bigotimes_\psi \mathcal E)$.

\item It is proved in \cite[Lemma 6.1]{MR2608451}
that every subproduct system is isomorphic to a standard subproduct
system. Therefore it is enough to consider standard subproduct
systems.
\end{enumerate}

\vspace{0.2in}

Let $T=(T_n)_{n \in \mathbb{Z_+}}$ be a completely contractive,
covariant  representation of a standard subproduct system
$X=(X(n))_{n \in \mathbb{Z_+}}$. We denote the positive operator
$(I_{\mathcal{H}}- \widetilde{T}_1 {\widetilde{T}_1^*} )^{1/2} \in
B(\mathcal{H})$ by $\triangle_*(T)$ and the defect space
$\overline{\mbox{\rm Im} ~\triangle_*(T)}$ by $\mathcal{D}$. It is
proved in \cite[Proposition 2.9]{MR2793449} that $\triangle_*(T) \in
\sigma(\mathcal{M})^\prime.$ Therefore $\mathcal{D}$ reduces
$\sigma(a)$ for each $a \in \mathcal{M}$. Thus using the reduced
representation $\sigma ^\prime $ we can form the tensor product of
the Hilbert space $\mathcal{D}$ with $X(n)$ for each $n \in
\mathbb{Z_+},$ and hence with $\mathcal{F}_X .$ For simplicity we
write $\sigma$ instead of $\sigma ^\prime.$ The {\it Poisson kernel}
of $T$ is the operator $K(T):\mathcal{H} \to \mathcal{F}_X
\bigotimes_{\sigma} \mathcal{D}$ defined by
\[
K(T)h:=\sum_{n \in \mathbb{Z_+}}(I_{X(n)}\otimes
\triangle_*(T)){\widetilde{T}_n}^*h \quad \quad (h \in \mathcal{H}).
\]

In the next proposition we recall the properties of the Poisson kernel from \cite{MR2793449}:

\bpp\label{prop2} Let $T=(T_n)_{n \in \mathbb{Z_+}}$ be a completely
contractive, covariant representation of a standard subproduct
system $X=(X(n))_{n \in \mathbb{Z_+}}$ over an $A^*$-algebra
$\mathcal M$. Then $K(T)$ is a contraction and
\[
K(T)^*(S_n(\zeta)\otimes I_{\mathcal{D}})=T_n(\zeta)K(T)^* \quad
\quad (n \in \mathbb{Z_+},~\zeta  \in X(n)).
\]
Moreover, $K(T)$ is an isometry if and only if $T$ is pure. \epp \bp
For each $h \in \mathcal{H}$, from (\ref{eqn4}) and (\ref{eqn6}) it
follows that
$$
\begin{array}{ccl}
\displaystyle\sum_{n \in \mathbb{Z_+}}\|(I_{X(n)} \otimes \triangle_*(T))
{\widetilde{T}_n}^*h\|^2 & = & \displaystyle\sum_{n \in \mathbb{Z_+}} \langle {\widetilde{T}_n}(I_{X(n)}\otimes \triangle_*(T)^2){\widetilde{T}_n}^*h,h\rangle
\\
& = &  \displaystyle\sum_{n \in \mathbb{Z_+}} \langle {\widetilde{T}_n}(I_{X(n)}\otimes (I_{\mathcal{H}}-{\widetilde{T}_1}{\widetilde{T}_1}^*) ){\widetilde{T}_n}^*h,h\rangle
\\
& = &  \displaystyle\sum_{n \in \mathbb{Z_+}} \langle
\widetilde{T}_n{\widetilde{T}_n}^*-\widetilde{T}_{n+1}{\widetilde{T}_{n+1}}^*h,h\rangle
\\
&=&\langle h,h \rangle- {\lim}_{n \rightarrow \infty}\langle
{\widetilde{T}_n}{\widetilde{T}_n}^*h,h\rangle,
\end{array}
$$
here we also used
${\widetilde{T}_0}{\widetilde{T}_0}^*=I_{\mathcal{H}}$. So $K(T)$ is
a well-defined contraction, and it is an isometry if $T$ is pure.
Now for each $n\in \mathbb{Z_+}$ and $z_n \in X(n)
\bigotimes_{\sigma} \mathcal{D}$ we have
\[
K(T)^*\left(\displaystyle\sum_{n \in \mathbb{Z_+}}
z_n\right)=\displaystyle\sum_{n \in \mathbb{Z_+}}
\widetilde{T}_n(I_{X(n)} \otimes \triangle_*(T))z_n.
\]
Therefore for every $m \in \mathbb{Z_+},$ $\eta \in X(m)$ and $h \in
\mathcal{D}$, (\ref{eqn4}) gives
$$
\begin{array}{ccl}
K(T)^*(S_n(\zeta)\otimes I_{\mathcal{D}})(\eta \otimes h) & = & K(T)^*(p_{n+m}(\zeta \otimes \eta)\otimes h) \\
& = &{\widetilde{T}_{n+m}}(p_{n+m}(\zeta \otimes \eta) \otimes \triangle_*(T)h)  \\
& = & {\widetilde{T}_n}(\zeta \otimes {\widetilde{T}_m}(\eta \otimes \triangle_*(T)h))\\
&=&T_n(\zeta)K(T)^*(\eta \otimes h).
\qedhere\end{array}
$$

\ep

\section{Invariant subspaces of covariant  representations}
In this section we first introduce the notion of invariant subspaces
for completely contractive, covariant representations and then in
Theorem \ref{thm1} we obtain a far reaching generalization of
\cite[Theorem 2.2]{MR3346131}.

Let $T=(T_n)_{n \in \mathbb{Z_+}}$ be a completely contractive,
covariant  representation of a standard subproduct system
$X=(X(n))_{n \in \mathbb{Z_+}}$ over an $A^*$-algebra $\mathcal M$.
A closed subspace $\mathcal{S}$ of $\mathcal{H}$ is called
\textit{invariant} for the covariant  representation $T$ if
$\mathcal{S}$ is invariant for $\sigma(\mathcal{M})$ and if
$\mathcal{S}$ is left invariant by each operator in the set $\{
T_n(\zeta) : \zeta  \in X(n), n \in \mathbb{N} \}$.

\bt\label{thm1} Let $T=(T_n)_{n \in \mathbb{Z_+}}$ be a pure
completely contractive, covariant  representation  of a standard
subproduct system $X=(X(n))_{n \in \mathbb{Z_+}}$ over an
$A^*$-algebra $\mathcal M$, and let $\mathcal S$ be a non-trivial
closed subspace of $\mathcal H$. Then  $\mathcal{S}$ is invariant
for $T$ if and only if there exist a Hilbert space $\mathcal{D},$ a
representation $\pi$ of $\mathcal M$ on $\mathcal D,$ and a partial
isometry $\Pi: \mathcal{F}_X\bigotimes_{\pi}\mathcal{D}\to
\mathcal{H} $ such that $\mathcal S = \mbox{ran~} \Pi$ and
$$
\Pi (S_n(\zeta)\otimes I_{\mathcal{D}})=T_n(\zeta)\Pi \quad \quad
(\zeta\in X(n), ~n\in \mathbb{Z_+}).
$$
\et \bp

Since $\mathcal{S}$ is invariant for $T=(T_n)_{n \in \mathbb{Z_+}}$,
we get a covariant representation $(V_n:={T_n}|_{\mathcal{S}})_{n
\in \mathbb{Z_+}}$ of the  standard subproduct system $X=(X(n))_{n
\in \mathbb{Z_+}}$ on $\mathcal S.$ We denote $V_0$ by $\pi$. Now
for each $ n \in \mathbb{N},~s\in \mathcal S,$ and $\zeta\in X(n)$,
$$
\langle \zeta\otimes s,\zeta\otimes s \rangle=\langle s,\pi(\langle
\zeta,\zeta\rangle) s\rangle=\langle s,\sigma(\langle
\zeta,\zeta\rangle)s\rangle=\langle \zeta\otimes s, \zeta\otimes s\rangle,
$$
yields an embedding $j_n$ from $X(n)\bigotimes_\pi \mathcal{S}$ into
$X(n)\bigotimes_{\sigma} \mathcal{H}.$ Thus for each $ n \in
\mathbb{N},$ $j_nj^*_n$ is an orthogonal projection.

\noindent For each $ n \in \mathbb{N}$, from the definition of the
map ${\widetilde{V}_n}: X(n)\bigotimes_{\pi} \mathcal{S} \to
\mathcal{S}$ it follows that
\[
{\widetilde{V}_n}(\zeta \otimes
s)=V_n(\zeta)s=T_n(\zeta)s=\widetilde{T}_n\circ j_n(\zeta\otimes s),
\]
for all $\zeta\in X(n)$ and $s\in\mathcal S$. It also follows that
\[
\langle\widetilde{V}_n\widetilde{V}^*_n
s,s\rangle=\langle\widetilde{T}_nj_nj^*_n\widetilde{T}^*_n
s,s\rangle\leq \langle\widetilde{T}_n\widetilde{T}^*_n s,s\rangle,
\]
for all $n\in \mathbb N$ and $s\in\mathcal S$. Hence the covariant
representation $V$ is pure as well as completely contractive.

\noindent Since the defect space $\mathcal{D}=\overline{\mbox{\rm
Im} ~\triangle_*(V)}$ of the representation $V$ is reducing for
$\pi$, it follows from Proposition \ref{prop2} that the Poisson
kernel $K(V): \mathcal{S} \to \mathcal{F}_X \bigotimes_\pi
\mathcal{D}$, defined by
$$
K(V)(s)=\displaystyle\sum_{n \in \mathbb{Z_+}}(I_{X(n)}\otimes
\triangle_*(V)){{\widetilde{V}_n}}^*s \quad \quad (s\in\mathcal S),
$$
is an isometry and
\[
K(V)^*(S_n(\zeta) \otimes I_{\mathcal{D}})=V_n(\zeta)K(V)^*,
\]
for all $n \in \mathbb{Z_+},$ and $\zeta\in X(n)$. Let $i_{\mathcal
S}:\mathcal S \to \mathcal{H}$ be the inclusion map. Clearly
$i_{\mathcal S}$ is an isometry and
\[
i_{\mathcal S}T_n(\cdot)|_{\mathcal{S}}=T_n(\cdot) i_{\mathcal S}.
\]
Therefore we get a map $\Pi:\mathcal{F}_X \bigotimes_{\pi}
\mathcal{D} \to \mathcal{H}$ defined by $\Pi:=i_{\mathcal S}K(V)^*
$. Then $$\Pi \Pi^*=i_{\mathcal S}K(V)^*(i_{\mathcal
S}K(V)^*)^*=i_{\mathcal S}i^*_{\mathcal S}=P_{\mathcal{S}},$$ the
projection on $\mathcal S$. Hence $\Pi$ is a partial isometry and
the range of $ \Pi$ is $\mathcal{S}$. From $i_{\mathcal
S}V_n=i_{\mathcal S}T_n|_{\mathcal{S}}=T_n i_{\mathcal S}$ and the
intertwining property of the Poisson kernel we deduce that
\[
\Pi(S_n(\zeta) \otimes I_{\mathcal{D}})=i_{\mathcal
S}K(V)^*(S_n(\zeta) \otimes I_{\mathcal{D}})=i_{\mathcal
S}V_n(\zeta)K(V)^*=T_n(\zeta) \Pi.
\]
Conversely, suppose that there exists a partial isometry $\Pi:
\mathcal{F}_X\bigotimes_{\pi}\mathcal{D}\to \mathcal{H} .$ Then $ran
\Pi$ is a closed subspace of $\mathcal{H}$ and the intertwining
relation for $\Pi$ implies that $ran \Pi$ is a $T=(T_n)_{n \in
\mathbb{Z_+}}$ invariant subspace of $\mathcal{H}$.
\ep

\begin{corollary}
Let $T=(T_n)_{n \in \mathbb{Z_+}}$ be a pure completely contractive,
covariant  representation  of a standard subproduct system
$X=(X(n))_{n \in \mathbb{Z_+}}$ over an $A^*$-algebra $\mathcal M$,
and $\mathcal S$ be a non-trivial closed subspace of $\mathcal
H.$ Then  $\mathcal{S}$ is invariant for $T$ if and only if there
exist a Hilbert space $\mathcal{D},$ a representation $\pi$ of
$\mathcal M$ on $\mathcal D,$ and a bounded linear operator $\Pi:
\mathcal{F}_X\bigotimes_{\pi}\mathcal{D}\to \mathcal{H} $ such that
$P_{\mathcal S}=\Pi\Pi^*,$ and
\[
\Pi (S_n(\zeta)\otimes I_{\mathcal{D}})=T_n(\zeta)\Pi \quad \quad
(\zeta\in X(n), ~n\in \mathbb{Z_+}).
\]
\end{corollary}

\begin{definition}
Let $X=(X(n))_{n \in \mathbb{Z_+}}$ be a standard subproduct system
over an $A^*$-algebra $\mathcal M$. Assume $\psi$ and $\pi$ to be
representations of $\mathcal M$ on Hilbert spaces $\mathcal E$ and
$\mathcal E',$ respectively.  A bounded operator $\Pi:
\mathcal{F}_X\bigotimes_{\pi}\mathcal{E'}\to \mathcal
F_X\bigotimes_\psi\mathcal E $ is called {\rm multi-analytic} if it
satisfies the following condition
$$\Pi (S_n(\zeta)\otimes I_{\mathcal{E'}})=(S_n(\zeta)\otimes I_{\mathcal E})\Pi~\mbox~{whenever}~\zeta\in X(n), ~n\in \mathbb{Z_+}.$$
Further we call it {\rm inner} if it is a partial isometry.
\end{definition}

As an application, we have the following Beurling-Lax-Halmos type
theorem (cf. \cite[Theorem 3.2]{MR2643314}) which extends
\cite[Theorem 2.4]{Po95} and \cite[Corollary 4.5]{J2}:

\begin{theorem}\label{Beur}
Assume $X=(X(n))_{n \in \mathbb{Z_+}}$ to be a standard subproduct
system over an $A^*$-algebra $\mathcal M$ and assume $\psi$ to be a
representation of $\mathcal M$ on a Hilbert space $\mathcal E$. Let
$\mathcal S$ be a non-trivial closed subspace of the Hilbert space
$\mathcal F_X\bigotimes_\psi\mathcal E.$ Then $\mathcal{S}$ is
invariant for $S\otimes I_{\mathcal E}$ if and only if there exist a
Hilbert space $\mathcal{E'},$ a representation $\pi$ of $\mathcal M$
on $\mathcal E',$ and an inner multi-analytic operator $\Pi:
\mathcal{F}_X\bigotimes_{\pi}\mathcal{E'}\to \mathcal
F_X\bigotimes_\psi\mathcal E $ such that $\mathcal S$ is the range
of $\Pi$.
\end{theorem}
\begin{proof}
Let $\mathcal{S}$ be an invariant subspace for $S\otimes I_{\mathcal
E}$. By Theorem \ref{thm1} we know that there exist a Hilbert space
$\mathcal{E'},$ a representation $\pi$ of $\mathcal M$ on $\mathcal
E',$ and a partial isometry $\Pi:
\mathcal{F}_X\bigotimes_{\pi}\mathcal{E'}\to \mathcal
F_X\bigotimes_\psi\mathcal E$ such that $\mathcal S = \mbox{ran~}
\Pi$ and
\[
\Pi (S_n(\zeta)\otimes I_{\mathcal{E'}})=(S_n(\zeta)\otimes
I_{\mathcal E})\Pi \quad \quad (\zeta\in X(n), ~n\in \mathbb{Z_+}).
\]
For the reverse direction, if we start with a partial isometry $\Pi:
\mathcal{F}_X\bigotimes_{\pi}\mathcal{E'}\to \mathcal
F_X\bigotimes_\psi\mathcal E$, then $ran \Pi$ is a closed subspace
of $\mathcal F_X\bigotimes_\psi\mathcal E$ and the intertwining
relation for $\Pi$ implies that $ran \Pi=\mathcal{S}$ is invariant
for $S\otimes I_{\mathcal E}.$
\end{proof}

For Beurling type classification in the tensor algebras setting see
also Muhly and Solel \cite[Theorem 4.7]{MS99}.

\section{Curvature}

The notion of a curvature for commuting tuples of row contractions
was introduced by Arveson \cite{A2000}.  This numerical invariant is
an analogue of the Gauss-Bonnet-Chern formula from Riemannian
geometry, and closely related to rank of Hilbert modules over
polynomial algebras. It has since been further analyzed by Popescu
\cite{MR1822685} (see also \cite{GP15} for recent results on a general class), Kribs \cite{DK} in the setting of noncommuting
tuples of operatos, and by Muhly and Solel \cite{MR2005882} in the
setting of completely positive maps on $C^*$ algebras of bounded
linear operators.

The purpose of this section is to study curvature for a more general
framework, namely, for completely contractive, covariant
representations of subproduct systems.

We begin by recalling the definition of left dimension
\cite{MR1473221} for a $W^*$-correspondences $E$ over a semifinite
factor $\mathcal{M}$ (see Muhly and Solel, Definition 2.5,
\cite{MR2005882}).

Let $\mathcal M$ be a semifinite factor and $\tau$ be a faithful
normal semifinite trace, and let $L^2(\mathcal M)$ be the GNS
construction for $\tau$. Note that for each $a\in\mathcal M$ there
exists a left multiplication operator, denoted by $\lambda (a)$, and
a right multiplication operator, denoted by $\rho (a)$, on
$L^2(\mathcal M)$. Each unital, normal, $*$-representation
$\sigma:\mathcal M\to B(\mathcal H)$ defines a {\it left $\mathcal
M$-module} $\mathcal H$. This yields an $\mathcal M$-linear isometry
$V:\mathcal H\to L^2(\mathcal M)\bigotimes l_2$. Here $\mathcal
M$-linear means
\[
V\sigma(a)=(\lambda(a)\otimes I_{l_2})V \quad \quad (a\in\mathcal M).
\]
Moreover
\[
V\sigma(\mathcal M)'V^*=p(\lambda(\mathcal M)\otimes
I_{l_2})'p\subseteq (\lambda(\mathcal M)\otimes I_{l_2})',
\]
where
\[
p:=VV^*\in (\lambda(\mathcal M)\otimes I_{l_2})',
\]
is a projection. One can observe that $(\lambda(\mathcal M)\otimes
I_{l_2})'$ equals the semifinite factor $\rho(\mathcal M)\bigotimes
B(l_2)$ whose elements can be written as matrices of the form
$(\rho(a_{ij}))$. For each positive element $x\in \sigma(\mathcal
M)'$, we express $VxV^*$ in the form $(\rho(a_{ij}))$, and define
$$tr_{\sigma(\mathcal M)'} (x):=\sum \tau (a_{ii}).$$ Note that
$tr_{\sigma(\mathcal M)'}$ is a faithful normal semifinite trace on
$\sigma(\mathcal M)'$. The {\it left dimension} of $\mathcal H$ is
defined by $$dim_{l}(\mathcal H):=tr_{\sigma(\mathcal M)'}(p).$$ For
each $W^*$-correspondence $E,$ the Hilbert space
$E\bigotimes_{\sigma} L^2(\mathcal M)$ has a natural left $\mathcal
M$-module structure. The left dimension of $E\bigotimes_{\sigma}
L^2(\mathcal M)$ will be denoted by $dim_l(E).$

Now let $X=(X(n))_{n \in \mathbb{Z_+}}$ be a standard subproduct
system of $W^*$-correspondences over a semifinite factor
$\mathcal{M}$. Let $T=(T_n)_{n \in \mathbb{Z_+}}$ be a completely
contractive, covariant representation of $X$ on a Hilbert space
$\mathcal{H}$. Define a contractive, normal and completely positive
map $\Theta_{T}:\sigma(\mathcal{M})'\to \sigma(\mathcal{M})'$ by
$$
\Theta_T(a):=\widetilde{T}_1(I_{E}\otimes a)\widetilde{T}^*_1,
$$
for all $a\in\sigma(\mathcal{M})'$. It follows from (\ref{eqn6}),
(\ref{eqn5}) and (\ref{eqn4}) that
\begin{align*}
\Theta^2_T(a)&=\Theta_T(\Theta_T(a))=\widetilde{T}_1(I_{E}\otimes (\widetilde{T}_1(I_{E}\otimes a)\widetilde{T}^*_1))\widetilde{T}^*_1
\\
&=\widetilde{T}_1(I_{E}\otimes \widetilde{T}_1)(I_{E^{\otimes 2}}\otimes a)(I_{E}\otimes \widetilde{T}^*_1)\widetilde{T}^*_1
\\
& = \widetilde{T}_2(p_2\otimes I_{\mathcal H})(I_{E^{\otimes
2}}\otimes a)(p^*_2\otimes I_{\mathcal H})\widetilde{T}^*_2
\\&=\widetilde{T}_2(I_{X(2)}\otimes a)\widetilde{T}^*_2,
\end{align*}
for all $a\in\sigma(\mathcal{M})'$. Inductively, we get
\begin{align*}
\Theta^n_T(a)&=\Theta_T(\Theta^{n-1}_T(a))=\widetilde{T}_1(I_{E}\otimes (\widetilde{T}_{n-1}(I_{X(n-1)}\otimes a)\widetilde{T}^*_{n-1}))\widetilde{T}^*_1
\\&=\widetilde{T}_1(I_{E}\otimes \widetilde{T}_{n-1})(I_{E}\otimes I_{X(n-1)}\otimes a)(I_{E}\otimes \widetilde{T}^*_{n-1})\widetilde{T}^*_1
\\&= \widetilde{T}_n(p_n\otimes I_{\mathcal H})(I_{E}\otimes I_{X(n-1)}\otimes a)(p^*_n\otimes I_{\mathcal H})\widetilde{T}^*_n
\\&=\widetilde{T}_n(I_{X(n)}\otimes a)\widetilde{T}^*_n,
\end{align*}
for all $a\in\sigma(\mathcal{M})'$ and $n\geq 2$.

The following is a reformulation of Muhly and Solel's result in our
setting \cite[Proposition 2.12]{MR2005882}:

\bpp\label{prop6.1} Let $X=(X(n))_{n \in \mathbb{Z_+}}$ be a
standard subproduct system of left-finite $W^*$-correspondences over a finite
factor $\mathcal{M}$. If $T=(T_n)_{n \in
\mathbb{Z_+}}$ is a completely contractive, covariant representation
of $X$ on $\mathcal{H}$ then
$$tr_{\sigma(\mathcal{M})'}(\Theta^n_{T}(x))\leq \|\widetilde{T}_n\|^2
\: dim_l(X(n))tr_{\sigma(\mathcal{M})'}(x),$$ for all $x \in
\sigma(\mathcal{M})_+'$. \epp

Let $X=(X(n))_{n \in {\mathbb{Z_+}}}$ be a standard subproduct
system of left-finite $W^*$-correspondences over a semifinite factor
$\mathcal{M}$. The \textit{curvature} of a completely contractive,
covariant representation $T=(T_n)_{n \in \mathbb{Z_+}}$ of $X$ on a
Hilbert space $\mathcal{H}$ is defined by
\begin{align}\label{eqn_curv}
Curv(T)=\lim_{k \to
\infty}\frac{tr_{\sigma(\mathcal{M})'}(I-\Theta_T^k(I))}{\sum_{j=0}^{k-1} dim_l(X(j))},
\end{align}
if the limit exists.

The following result is well known (cf. Popescu
\cite[p.280]{MR1822685}).

\bl \label{lemp} Let $\{a_j\}_{j=0}^{\infty}$ and
$\{b_j\}_{j=0}^{\infty}$ be two real sequences, and let $a_j \geq 0$
and $b_j>0$ for all $j \geq 0$. Consider the partial sums
$A_k:=\sum_{j=0}^{k-1}a_j$ and $B_k:=\sum_{j=0}^{k-1}b_j$, and
suppose that $B_k \to \infty$ as $k \to \infty$. Then
\[
\lim_{k \to \infty}\frac{A_k}{B_k}=L,
\]
whenever $L:=\lim_{j \to \infty}\frac{a_j}{b_j}$ exists.
\el

Coming back to our definition of curvatures, we note that
\[
\begin{split}
tr_{\sigma(\mathcal{M})'}(I-\Theta_T^k(I)) & = \displaystyle
\sum^{k-1}_{j=0}
tr_{\sigma(\mathcal{M})'}\Theta_T^j(I-\Theta_T(I))
\\
& = \displaystyle \sum^{k-1}_{j=0}
tr_{\sigma(\mathcal{M})'}\Theta_T^j(\triangle_*(T)^2).
\end{split}
\]
From this, and our previous lemma, it follows that $Curv(T)$ is well
defined whenever the following two conditions are satisfied:
\begin{itemize}
\item [(1)] $
\lim_{j \to
\infty}\frac{tr_{\sigma(\mathcal{M})'}\Theta_T^j(\triangle_*(T)^2)}{
dim_l(X(j))}$ exists, and
\item [(2)] $
\lim_{k \to \infty}  \sum_{j=0}^{k-1}dim_l(X(j))=\infty$.
\end{itemize}

The next result concerns the existence of curvatures in the setting
of completely contractive, covariant representations on product
systems (cf. \cite[Example 1.2]{MR2608451}). This is an analogue of
the result by Muhly and Solel \cite[Theorem 3.3]{MR2005882}. The
curvature for completely contractive, covariant representation of
the standard subproduct system (see Example \ref{eg-3}) will be
discussed at the end of this section.

\bt Let $X=(X(n))_{n \in \mathbb{Z_+}}$ be a product system of
$W^*$-correspondences over a finite factor $\mathcal{M}$, that is,
$X(n)=E^{\otimes n}$ where $E:=X(1)$ is a left-finite
$W^*$-correspondence. Set $d: =dim_l(E)$. If $T=(T_n)_{n \in
\mathbb{Z_+}}$ is a completely contractive, covariant representation
of $X$ on $\mathcal{H}$, then the following holds:
\begin{itemize}
\item [(1)] The limit in the definition of $Curv(T)$ exists, either
as a positive number or $+ \infty$.
\item [(2)] $Curv(T)=\infty$ if and only if $tr_{\sigma(\mathcal{M})'}(I-\Theta_T(I))=\infty$.
\item [(3)] If $tr_{\sigma(\mathcal{M})'}(I-\Theta_T(I)) < \infty$, then $Curv(T) < \infty$. Moreover, in this case we have the following:
\begin{enumerate}
\item [(3a)] for $d\geq1$ we have
\[
Curv(T)=\lim_{k \to \infty}
\frac{tr_{\sigma(\mathcal{M})'}(\Theta_{T}^k(I)-\Theta_{T}^{k+1}(I))}{d^k},
\]
in particular if $d>1$, then we further get
\[
Curv(T)=(d-1)\lim_{k \to
\infty}\frac{tr_{\sigma(\mathcal{M})'}(I-\Theta_{T}^k(I))}{d^k};
\]
\item [(3b)] for $d<1$, $\lim_{k \to \infty}tr_{\sigma(\mathcal{M})'}(I-\Theta_{T}^k(I)) < \infty$, and
\[
Curv(T)=(1-d)\left(\lim_{k \to
\infty}tr_{\sigma(\mathcal{M})'}(I-\Theta_{T}^k(I))\right).
\]
\end{enumerate}
\end{itemize}
\et \bp From \cite[Theorem
3.3]{MR2005882} it follows that $dim_l(X(j))=d^j.$ Let $ a_k=tr_{\sigma(\mathcal{M})'}(\Theta_{T}^k(I)-\Theta_{T}^{k+1}(I))$ for $k\geq
0$. Then Proposition \ref{prop6.1} yields
$$\begin{array}{cll}
a_{k+1} & = &tr_{\sigma(\mathcal{M})'}(\Theta_{T}( \Theta_{T}^k(I)-\Theta_{T}^{k+1}(I) ))\\
& \leq & \|\widetilde{T}_1\|^2 dim_l(E) tr_{\sigma(\mathcal{M})'}( \Theta_{T}^k(I)-\Theta_{T}^{k+1}(I) )\\
& \leq& da_k,
\end{array}
$$
for all $k\geq 0$. If $a_0=\infty$, then the fact that
$\{\widetilde{T}_n\widetilde{T}_n^*\}_{n \in \mathbb{Z_+}}$ is a
decreasing sequence of positive contractions implies that
\[
tr_{\sigma(\mathcal{M})'}(I-\Theta_T^k(I))=\infty \quad \quad (k
\geq 0).
\]
If $a_0<\infty,$ then $\{\frac{a_j}{d^j}\}^{\infty}_{j=0}$ is a
non-increasing sequence of non-negative numbers. Then $0\leq L \leq
a_0$ where $L := \lim \frac{a_j}{d^j}$.

\noindent Let $d\geq1$. Since
\[
tr_{\sigma(\mathcal{M})'}(I-\Theta_{T}^k(I))=\sum_{j=0}^{k-1}a_j,
\]
by Lemma \ref{lemp} (for $b_j=d^j$) the limit defining $Curv(T)$
exists and $Curv(T) = L$.

\noindent Now let  $d>1$. Then
$\sum_{j=0}^{k-1}d^j=\frac{d^k-1}{d-1}$ and $\lim_{k \to
\infty}\frac{d^k-1}{d^k}=1$ yields
$$
\begin{array}{cll}
Curv(T) & = & \lim_{k \to \infty }\frac{tr_{\sigma(\mathcal{M})'}(I-\Theta_{T}^k(I)) }{\frac{d^k-1}{d-1}}
\\
& = & (d-1)\lim_{k \to
\infty}\frac{tr_{\sigma(\mathcal{M})'}(I-\Theta_{T}^k(I))}{d^k-1}
\lim_{k \to \infty}\frac{d^k-1}{d^k}
\\
& = & (d-1)\lim_{k \to
\infty}\frac{tr_{\sigma(\mathcal{M})'}(I-\Theta_{T}^k(I))}{d^k}.
\end{array}
$$
This proves statement $(3a)$.

\noindent Finally, let $d<1$ so that $\sum_{j=0}^{\infty} d^j =
1/(1-d)$. Since $a_j \leq d^j a_0$ for all $j\geq 0$, $\lim_{k \to
\infty}tr_{\sigma(\mathcal{M})'}(I-\Theta_{T}^k(I))$ exists and is
finite. This completes the proof of $(3b)$. The proof of statements
$(1)$ and $(2)$ follows by noting that whenever $a_0$ is finite, the
limit defining $Curv(T)$ exists and is finite.  \ep

Recall that
\[
\Theta_T(x)=\widetilde{T_1}(I_{E} \otimes x)\widetilde{T_1}^*,
\]
for all $x\in\sigma(\mathcal{M})'$, and
\[
Q=\lim_{n \to \infty}\widetilde{T}_n\widetilde{T}_n^*=\lim_{n \to
\infty}\Theta_{T}^n(I_{\mathcal{H}}).
\]
Using the intertwining property of the Poisson kernel
\[
K(T)^*(S_{n}(\zeta)\otimes I_{\mathcal{D}})=T_n(\zeta)K(T)^*,
\]
we have
$$
\begin{array}{cll}
\widetilde{T}_n(I_{X(n)} \otimes K(T)^*)(\zeta \otimes k) & = & \widetilde{T}_n(\zeta \otimes K(T)^*k) \\
& = & T_n(\zeta)K(T)^* k \\
& = & K(T)^*(S_n(\zeta) \otimes I_{\mathcal{D}})k,
\end{array}
$$
for all $\zeta \in X(n), k\in \mathcal{F}_X \bigotimes_{\sigma}
\mathcal{D}, n \in \mathbb{Z_+}$. Then
$$
\widetilde{T}_n(I_{X(n)} \otimes
K(T)^*)=K(T)^*\widetilde{(S_n(\cdot) \otimes I_{\mathcal{D}})},
$$
and hence $\Theta_T^{n}(Q)=Q$ and
\[
K(T)^*K(T)=I_{\mathcal{H}}-Q,
\]
yields
\[
\begin{split}
K(T)^* & (I_{\mathcal{F}_X \bigotimes_{\sigma}
\mathcal{D}}-\Theta_{S \otimes I_{\mathcal{D}}}^n (I_{\mathcal{F}_X
\bigotimes_{\sigma} \mathcal{D}}))K(T)
\\
& = K(T)^*K(T)- K(T)^*(\widetilde{S_n(\cdot) \otimes
I_{\mathcal{D}}})(\widetilde{S_n(\cdot) \otimes
I_{\mathcal{D}}})^*K(T)
\\
& = K(T)^*K(T)-\widetilde{T}_n(I_{X(n)} \otimes K(T)^*)(\widetilde{S_n(\cdot) \otimes I_{\mathcal{D}}})^* K(T)
\\
& = I_{\mathcal{H}}-Q-\widetilde{T}_{n}(I_{X(n)} \otimes K(T)^*)(I_{X(n)} \otimes K(T)) \widetilde{T}_{n}^*
\\
& = I_{\mathcal{H}}-Q-\widetilde{T}_{n}(I_{X(n)} \otimes K(T)^*K(T)) \widetilde{T}_{n}^*
\\
& =I_{\mathcal{H}}-Q- \widetilde{T}_{n}(I_{X(n)} \otimes (I_\mathcal{H}-Q)) \widetilde{T}_{n}^*
\\
& = I_{\mathcal{H}}-Q-\Theta_T^{n}(I_\mathcal{H}-Q)
\\
& = I_{\mathcal{H}}-\Theta_T^{n}(I_\mathcal{H}).
\end{split}
\]

Therefore one can compute the curvature, in terms of Poisson kernel,
in the following sense:

\begin{proposition}\label{proposition1}
Let $X=(X(n))_{n \in {\mathbb{Z_+}}}$ be a standard subproduct
system of left-finite $W^*$-correspondences over a finite factor $\mathcal{M}$. If $T=(T_n)_{n \in \mathbb{Z_+}}$ is a completely
contractive, covariant representation of $X$ on a Hilbert space
$\mathcal{H}$, then the curvature of $T$ is given by
\begin{align}\label{eqn2} &~~~~~~~~Curv(T)\\&\nonumber =\lim_{k \to
\infty}\frac{tr_{\sigma(\mathcal{M})'}(K(T)^*(I_{\mathcal{F}_X
\bigotimes_{\sigma} \mathcal{D}}-\Theta_{S \otimes
I_{\mathcal{D}}}^k(I_{\mathcal{F}_X \bigotimes_{\sigma}
\mathcal{D}}))K(T))}{\sum_{j=0}^{k-1}dim_l(X(j))},
\end{align}
if the limit exists.

\end{proposition}
The following theorem generalizes \cite[Theorem 2.1]{MR1822685}.

\begin{theorem}\label{theorem1}
Let $T=(T_n)_{n \in \mathbb{Z_+}}$ to be a completely contractive,
covariant  representation  of a standard subproduct system
$X=(X(n))_{n \in \mathbb{Z_+}}$  of  $A^*$-correspondences over an
$A^*$-algebra $\mathcal{M}$. Then there exist a Hilbert space
$\mathcal E,$ a representation $\psi$ of $\mathcal M$ on $\mathcal
E,$ and an inner multi-analytic operator $\Pi:\mathcal
F_X\bigotimes_\psi\mathcal E
\to\mathcal{F}_X\bigotimes_{\sigma}\mathcal{D} $ such that
$$
I_{\mathcal{F}_X\bigotimes_{\sigma}\mathcal{D}}-K(T)K(T)^*=\Pi\Pi^*.
$$
\end{theorem}
\begin{proof}
Proposition \ref{prop2} implies that $(ran K(T))^\perp$ is invariant
for the covariant representation $S\otimes I_{\mathcal D}.$ Now we
use Theorem \ref{Beur} and obtain a Hilbert space $\mathcal E,$ a
representation $\psi$ of $\mathcal M$ on $\mathcal E,$ and a partial
isometry $\Pi:\mathcal F_X\bigotimes_\psi\mathcal E
\to\mathcal{F}_X\bigotimes_{\sigma}\mathcal{D} $ such that $(ran
K(T))^\perp$ is the range of $\Pi,$ and
$$
\Pi(S_n(\zeta)\otimes I_{\mathcal E})=(S_n(\zeta)\otimes
I_{\mathcal{D}})\Pi,
$$
for all $\zeta\in X(n)$ and $n\in \mathbb{Z_+}$. Finally, using the
fact that $\Pi$ is a partial isometry and
\[
(ran K(T))^\perp
=ran(I_{\mathcal{F}_X\bigotimes_{\sigma}\mathcal{D}}-K(T)K(T)^*),
\]
we get the desired formula.
\end{proof}

The following is an analogue of \cite[Theorem 3.32]{MR2643314} in
our context in terms of multi-analytic operators.

\bt
Let $X=(X(n))_{n \in {\mathbb{Z_+}}}$ be a standard subproduct
system of left-finite $W^*$-correspondences over a finite factor $\mathcal{M}$. If $T=(T_n)_{n \in \mathbb{Z_+}}$ is a completely
contractive, covariant representation of $X$ on a Hilbert space
$\mathcal{H}$, and
\begin{align}\label{eq1.2}
tr_{(\phi_{\infty}(\mathcal{M})\otimes I_{\mathcal
D})'}(I_{\mathcal{F}_X \bigotimes_{\sigma} \mathcal{D}}-\Theta^k_{S
\otimes I_{\mathcal{D}}}(I_{\mathcal{F}_X \bigotimes_{\sigma}
\mathcal{D}})) < \infty,
\end{align}
for all $k \geq 1$, then there exist a Hilbert space $\mathcal E,$ a
representation $\psi$ of $\mathcal M$ on $\mathcal E,$ and an inner
multi-analytic operator $\Pi:\mathcal F_X\bigotimes_\psi\mathcal E
\to\mathcal{F}_X\bigotimes_{\sigma}\mathcal{D} $ such that
\begin{align*}
&~~~~Curv(T)\\&=\lim_{k \to \infty}\frac{
tr_{(\phi_{\infty}(\mathcal{M})\otimes I_{\mathcal
D})'}((I_{\mathcal{F}_X\bigotimes_{\sigma}\mathcal{D}}-\Pi\Pi^*)(I_{\mathcal{F}_X
\bigotimes_{\sigma} \mathcal{D}}-\Theta_{S \otimes
I_{\mathcal{D}}}^k(I_{\mathcal{F}_X \bigotimes_{\sigma}
\mathcal{D}})))}{\sum_{j=0}^{k-1}dim_l(X(j))}.
\end{align*}
\et
\begin{proof}
For simplicity of notation we use $I$ for $I_{\mathcal{F}_X
\bigotimes_{\sigma} \mathcal{D}}$ and also use $\Theta $ for
$\Theta_{S \otimes I_{\mathcal{D}}}$. Define a representation $\rho$
of $\mathcal M$ on $\mathcal H\bigoplus \left(\mathcal{F}_X
\bigotimes_{\sigma} \mathcal{D}\right)$ by
\[
\rho(a)=\begin{pmatrix}
{\sigma(a)} & {0}  \\
{0} & \phi_{\infty}(a)\otimes I_{\mathcal D}
\end{pmatrix},
\]
for all $a \in \mathcal M$. Then
\begin{align}
\nonumber & tr_{\sigma(\mathcal{M})'}(K(T)^*(I-\Theta^k(I))K(T))
\\
\nonumber&=tr_{\rho(\mathcal{M})'}\begin{pmatrix}
{K(T)^*(I-\Theta^k(I))K(T)} & {0}  \\
{0} & 0
\end{pmatrix}
\\
\nonumber&=tr_{\rho(\mathcal{M})'}\left(\begin{pmatrix}
{0} & {K(T)^*(I-\Theta^k(I))^{\frac{1}{2}})}  \\
    {0} & 0
    \end{pmatrix}
    \begin{pmatrix}
    {0} & {0}  \\
    {(I-\Theta^k(I))^{\frac{1}{2}}K(T)} & 0
    \end{pmatrix}\right)
    \\ \nonumber&=tr_{\rho(\mathcal{M})'}\left( \begin{pmatrix}
    {0} & {0}  \\
    {(I-\Theta^k(I))^{\frac{1}{2}}K(T)} & 0
    \end{pmatrix}\begin{pmatrix}
    {0} & {K(T)^*(I-\Theta^k(I))^{\frac{1}{2}})}  \\
    {0} & 0
    \end{pmatrix}
\right)
    \\ \nonumber&=tr_{\rho(\mathcal{M})'}\begin{pmatrix}
    {0} & {0}  \\
    {0} & (I-\Theta^k(I))^{\frac{1}{2}}K(T)K(T)^*(I-\Theta^k(I))^{\frac{1}{2}}
    \end{pmatrix}
        \\ \label{eqna1}&=  tr_{(\phi_{\infty}(\mathcal{M})\otimes I_{\mathcal D})'}((I-\Theta^k(I))^{\frac{1}{2}}K(T)K(T)^*(I-\Theta^k(I))
        ^{\frac{1}{2}}).
    \end{align}
Now by Theorem \ref{theorem1}, there exist a Hilbert space $\mathcal
E,$ a representation $\psi$ of $\mathcal M$ on $\mathcal E$, and an
inner multi-analytic operator $\Pi:\mathcal
F_X\bigotimes_\psi\mathcal E
\to\mathcal{F}_X\bigotimes_{\sigma}\mathcal{D}$ such that
\begin{align*}
&~~~~Curv(T)\\&
=\lim_{k \to \infty}\frac{
tr_{(\phi_{\infty}(\mathcal{M})\otimes I_{\mathcal
D})'}((I_{\mathcal{F}_X\bigotimes_{\sigma}\mathcal{D}}-\Pi\Pi^*)(I_{\mathcal{F}_X
\bigotimes_{\sigma} \mathcal{D}}-\Theta_{S \otimes
I_{\mathcal{D}}}^k(I_{\mathcal{F}_X \bigotimes_{\sigma}
\mathcal{D}})))}{\sum_{j=0}^{k-1}dim_l(X(j))}.
\end{align*}

Then equation (\ref{eqna1}) and Proposition \ref{proposition1} yields
\begin{align*}
Curv(T)
=&\lim_{k \to \infty}\frac{tr_{\sigma(\mathcal{M})'}(K(T)^*(I-\Theta^k(I))K(T))}{\sum_{j=0}^{k-1}dim_l(X(j))}
\\
=&\lim_{k \to \infty}\frac{ tr_{(\phi_{\infty}(\mathcal{M})\otimes
I_{\mathcal
D})'}((I-\Theta^k(I))^{\frac{1}{2}}K(T)K(T)^*(I-\Theta^k(I))^{\frac{1}{2}})}{\sum_{j=0}^{k-1}dim_l(X(j))}
\\
= &\lim_{k \to \infty}\frac{ tr_{(\phi_{\infty}(\mathcal{M})\otimes I_{\mathcal D})'}
((I-\Theta^k(I))^{\frac{1}{2}}(I-\Pi\Pi^*)(I-\Theta^k(I))^{\frac{1}{2}})}{\sum_{j=0}^{k-1}dim_l(X(j))}
\\
=&\lim_{k \to \infty}\frac{ tr_{(\phi_{\infty}(\mathcal{M})\otimes
I_{\mathcal
D})'}((I-\Pi\Pi^*)(I-\Theta^k(I)))}{\sum_{j=0}^{k-1}dim_l(X(j))}.
\end{align*}
The third equality follows from the observation that: since
$tr_{(\phi_{\infty}(\mathcal{M})\otimes I_{\mathcal
D})'}(I-\Theta^k(I))$ is finite, $(I-\Theta^k(I))^{\frac{1}{2}}$
belongs to the ideal $$\{x:tr_{(\phi_{\infty}(\mathcal{M})\otimes
I_{\mathcal D})'}(x^*x)<\infty\},$$ and hence
$tr_{(\phi_{\infty}(\mathcal{M})\otimes I_{\mathcal
D})'}(I-\Theta^k(I))^{\frac{1}{2}}<\infty$, for all $k$. \qedhere
\end{proof}

\begin{remark}
Consider the standard subproduct system of Example \ref{eg-3}, and
let $\mbox{dim~} X(1) = d<\infty$. It is easy to verify that
\begin{align*}
dim_l(X(j))&= \binom{j+d-1}{j},
\end{align*}
for all $j \geq 1$. By induction it follows that
\begin{align*}
\displaystyle \sum^{k-1}_{j=0} dim_l(X(j))&=\frac{k(k+1)\ldots
(k+d-1)}{d!}\sim\frac{k^d}{d!} \quad \quad (k \geq 1).
\end{align*}
Therefore, in this case, our curvature defined in (\ref{eqn_curv})
coincides with the Arveson's curvature for the row contraction
\[
(T_1(e_1),T_1(e_2),\ldots,T_1(e_d))
\]
(under the finite rank assumption, that is, $rank~(I_{\mathcal H} -
\sum_{i=1}^d T_1(e_i) T_1(e_i)^*) < \infty$) on the Hilbert space
$\mathcal H$ (cf. \cite[Theorem C]{A2000}). One can also compare the
curvature obtained in (\ref{eqn2}) with Popescu's curvature (cf.
\cite[Equation 2.11]{MR1822685}). Furthermore, observe that the
condition (\ref{eq1.2}) is automatic for finite rank row
contractions.
\end{remark}

\section{Wandering subspaces}

The notion of wandering subspaces of bounded linear operators on
Hilbert spaces was introduced by Halmos \cite{PH}. With this as a
motivation we extend the notion of wandering subspace (cf. \cite[p.
561]{H16}) for covariant representations of standard subproduct
systems, as follows: Let $T=(T_n)_{n \in \mathbb{Z_+}}$ be a
covariant  representation of a standard subproduct system
$X=(X(n))_{n \in \mathbb{Z_+}}$ over an $A^*$-algebra $\mathcal M$.
A closed subspace $\mathcal{S}$ of $\mathcal{H}$ is
called \textit{wandering} for the covariant  representation $T$ if
it is $\sigma(\mathcal M)$-invariant, and if for each $n \in \mathbb{N}$ the subspace $\mathcal S$ is orthogonal to
\[
\mathfrak{L}_n(\mathcal{S},T):=\bigvee \{T_n(p_n(\zeta))s\: : \:
\zeta \in E^{\otimes n},s \in \mathcal{S}\}.
\]
When there is no confusion we use the notation
$\mathfrak{L}_n(\mathcal{S})$ for $\mathfrak{L}_n(\mathcal{S},T),$
and also use $\mathfrak{L}(\mathcal{S})$ for
$\mathfrak{L}_1(\mathcal{S})$. A wandering subspace $\mathcal W$ for $T$ is called
\textit{generating} if $\mathcal H=\overline{\mbox{span}}
\{\mathfrak L_n (\mathcal W):n\in\mathbb Z_+\}$.

In the following proposition we prove that the wandering subspaces
are naturally associated with invariant subspaces of covariant
representations of standard subproduct systems.

\bpp \label{prop5.1} Let $T=(T_n)_{n \in \mathbb{Z_+}}$ be a covariant
representation of a standard subproduct system $X=(X(n))_{n \in
\mathbb{Z_+}}$ over an $A^*$-algebra $\mathcal M$. If $\mathcal S$
is a closed $T$-invariant subspace of $\mathcal{H},$ then
$\mathcal{S}\ominus\mathfrak{L}(\mathcal{S})$ is a wandering
subspace for $T|_{\mathcal S}:=({T_n}|_{\mathcal S})_{n\in
\mathbb{Z_+}}.$ \epp

\bp Let $n \geq 1$ and $\eta=\xi_1\otimes\xi_{n-1} \in E^{\otimes
n}$ for some $\xi_1 \in E$ and $\xi_{n-1}\in E^{\otimes n-1}$. Let
$x, s \in \mathcal{S}\ominus \mathfrak{L}(\mathcal{S})$ so that $
y=T_n(p_n(\eta))s \in
\mathfrak{L}_n(\mathcal{S}\ominus\mathfrak{L}(\mathcal{S}))$. Then
\[
\begin{split}
\langle x, y \rangle & = \langle x ,T_n(p_n(\eta))s\rangle
\\
& =\langle x, T_n(p_n(\xi_1 \otimes \xi_{n-1}))s \rangle
\\
& =\langle x,T_{1+(n-1)}(p_{1+(n-1)}(\xi_1 \otimes \xi_{n-1}))s\rangle \\
& = \langle x,T_1(\xi_1)T_{n-1}(\xi_{n-1})s\rangle
\\
& =0,
\end{split}
\]
since $\mathcal{S}$ in invariant under $T_{n-1}(\xi_{n-1})$. Therefore $
\mathcal{S}\ominus \mathfrak{L}(\mathcal{S})$ is orthogonal to
$\mathfrak{L}_n(\mathcal{S}\ominus\mathfrak{L}(\mathcal{S}))$, $n
\geq 1$ and hence $\mathcal{S}\ominus\mathfrak{L}(\mathcal{S})$ is a
wandering subspace for $T|_{\mathcal S}=({T_n}|_{\mathcal S})_{n\in
\mathbb{Z_+}}$.  \ep

Let $T=(T_n)_{n \in \mathbb{Z_+}}$ be a covariant  representation of
a standard subproduct system $X=(X(n))_{n \in \mathbb{Z_+}}$.
Suppose $\mathcal W$ is a wandering subspace for $T.$ Set
\[
\mathcal G_{T,\mathcal W}:=\bigvee_{n\in\mathbb Z_+}\mathfrak
L_n(\mathcal W).
\]
Note that
\begin{align*}
&\mathfrak L\left(\bigvee_{n\in\mathbb Z_+}\mathfrak L_n(\mathcal W)\right) \\& =
\overline{\mbox{span}}\{T_1(p_1(\zeta))T_n(p_n(\eta))w:\zeta\in E, \eta\in E^{\otimes n},w\in \mathcal W, n\in\mathbb Z_+\}
\\
& =  \overline{\mbox{span}}\{T_{n+1}(p_{n+1}(p_1(\zeta)\otimes
p_n(\eta))w:\zeta\in E, \eta\in E^{\otimes n}, w\in \mathcal W,
n\in\mathbb Z_+\}\\ &\subset \bigvee_{n\in\mathbb N}\mathfrak
L_n(\mathcal W).
\end{align*}
In the other direction, we have
\begin{align*}
&\bigvee_{n\in\mathbb N}\mathfrak L_n(\mathcal W)
\\&= \overline{\mbox{span}}\{T_{n}(p_{n}(p_1(\zeta)\otimes p_{n-1}(\eta))w:\zeta\in E, \eta\in E^{\otimes {n-1}},w\in \mathcal W, n\in\mathbb N\}
\\
& =  \overline{\mbox{span}}\{T_{1}(p_1(\zeta))T_{n-1}( p_{n-1}(\eta))w:\zeta\in E, \eta\in E^{\otimes {n-1}},w\in
\mathcal W, n\in\mathbb N\}
\\
&\subset \mathfrak L\left(\bigvee_{n\in\mathbb Z_+}\mathfrak
L_n(\mathcal W)\right).
\end{align*}
Thus these sets are equal, and it follows that
\begin{align*}
\mathcal G_{T,\mathcal W}\ominus \mathfrak L(\mathcal  G_{T,\mathcal
W})=\bigvee_{n\in\mathbb Z_+}\mathfrak L_n(\mathcal W)\ominus
\mathfrak L\left(\bigvee_{n\in\mathbb Z_+}\mathfrak L_n(\mathcal
W)\right)=\mathcal W.
\end{align*}

Hence we have the following uniqueness result:

\begin{proposition}\label{prop6}
Let $T=(T_n)_{n \in \mathbb{Z_+}}$ be a covariant  representation of
a standard subproduct system $X=(X(n))_{n \in \mathbb{Z_+}}$ over an
$A^*$-algebra $\mathcal M$. If $\mathcal W$ is a wandering subspace
for $T,$ then
$$\mathcal W=\mathcal  G_{T,\mathcal W}\ominus \mathfrak L(\mathcal  G_{T,\mathcal W}).$$
Moreover, if $\mathcal W$ is also generating, then $\mathcal
W=\mathcal H\ominus \mathfrak L(\mathcal H).$
\end{proposition}
In Theorem \ref{thm1} we observed that each non-trivial closed
subspace $\mathcal S\subset\mathcal H,$ which is invariant under a
pure completely contractive, covariant  representation $T=(T_n)_{n
\in \mathbb{Z_+}}$ of a standard subproduct system $X=(X(n))_{n \in
\mathbb{Z_+}},$ can be written as $\mathcal
S=\Pi(\mathcal{F}_X\bigotimes_{\pi}\mathcal{D}).$ In the following
theorem we study wandering subspaces in a general situation when $T$
is not necessarily pure.

\bt \label{Theorem 5.3} Let $X=(X(n))_{n \in \mathbb{Z}_+}$ be a standard subproduct
system over an $A^*$-algebra $\mathcal M$. Let $\pi:\mathcal M\to
B(\mathcal E)$ be a representation on a Hilbert space $\mathcal E$
and  $T=(T_n)_{n \in \mathbb{Z}_+}$ be the covariant representation
of $X$. Let $\Pi: \mathcal{F}_X \bigotimes_\pi \mathcal{E} \to
\mathcal{H}$ be a partial isometry such that $\Pi(S_n(\zeta)\otimes
I_{\mathcal{E}})=T_n(\zeta)\Pi$ for every $\zeta \in X(n), n \in
\mathbb{Z}_+$. Then $\mathcal{S}:=\Pi(\mathcal{F}_X \bigotimes_\pi
\mathcal{E})$ is a closed $T$-invariant subspace,
$\mathcal{W}:=\mathcal{S} \ominus \mathfrak{L}(\mathcal{S})$ is a
wandering subspace for $T|_{\mathcal{S}},$ and $\mathcal{W}=\Pi((ker
\Pi)^{\bot}\bigcap \mathcal{M}\bigotimes_\pi \mathcal{E})$. \et

\bp Define $F=(ker \Pi)^{\bot}\bigcap \mathcal{M}\bigotimes_\pi
\mathcal{E}$. Since $\mathcal{S}$ is the range of $\Pi,$ it is a
closed $T$-invariant subspace. Therefore by
Proposition \ref{prop5.1},  the subspace $\mathcal{W}$ is a wandering subspace for
$T|_{\mathcal{S}}$.
$$
\begin{array}{cl}
&\mathfrak{L}(\mathcal S, T) \\& = \mathfrak{L}(\Pi(\mathcal{F}_X
\bigotimes_\pi \mathcal{E}),T)
\\& =  \bigvee\{ \: T_1(\zeta)k \:\: : \:\: k \in\Pi(\mathcal{F}_X \bigotimes_\pi \mathcal{E}),\zeta\in X(1) \} \\
& = \bigvee\{ \: T_1(\zeta)\Pi(l) \:\: : \:\: l \in \mathcal{F}_X \bigotimes_\pi \mathcal{E},\zeta\in X(1) \}\\
& = \bigvee\{ \: \Pi(S_1(\zeta) \otimes I_{\mathcal{E}})(l_m\otimes
e) \:\: : \:\: l_m\otimes e \in X(m) \bigotimes_\pi
\mathcal{E},\zeta\in X(1), m \in \mathbb{Z}_+ \}.
\end{array}
$$
For $x \in (ker \Pi)^{\bot}\bigcap  \mathcal{M}\bigotimes_\pi
\mathcal{E}$ and $ l_m\otimes e \in X(m) \bigotimes_\pi \mathcal{E}$
we have
$$
\begin{array}{ll}
\langle \Pi x,\Pi(S_1(\zeta) \otimes I_{\mathcal{E}})(l_m\otimes e) \rangle & = \langle \Pi ^*\Pi x,(S_1(\zeta) \otimes I_{\mathcal{E}})(l_m\otimes e)\rangle
\\
& = \langle  x,(S_1(\zeta) \otimes I_{\mathcal{E}})(l_m\otimes
e)\rangle\\&= \langle  x,P_{1+m}(\zeta\otimes l_m) \otimes e\rangle
\\ & =0,
\end{array}
$$
and hence $\Pi((ker \Pi)^{\bot}\bigcap \mathcal{M}\bigotimes_\pi
\mathcal{E})\subset\mathcal W$.

\noindent For the converse direction, let $x \in \mathcal{S} \ominus
\mathfrak{L}(\mathcal{S},T)=\mathcal W$, and $\Pi(y)=x$ for some
$y\in (ker\Pi)^{\bot}$. Therefore for any $\zeta \in X(1),\eta
\otimes e\in \mathcal{F}_X\bigotimes_\pi \mathcal E$ we have
\begin{eqnarray}\label{err1}
\langle y,(S_1(\zeta)\otimes I_{\mathcal{E}})(\eta \otimes
e)\rangle&=&\langle \Pi y,\Pi(S_1(\zeta)\otimes
I_{\mathcal{E}})(\eta \otimes e)\rangle=0.
\end{eqnarray}
Recall that by definition we have
\begin{align*}
&\mathfrak L(\mbox{$\mathcal{F}_X
\bigotimes_\pi \mathcal{E}$},S\otimes I_{\mathcal
E})\\&=\bigvee\{(S_1(\zeta) \otimes I_{\mathcal{E}})(\eta \otimes e) :
\eta \in X(m), \zeta \in X(1), e \in \mathcal{E}, m \in
\mathbb{Z}_+\}.
\end{align*}
Since $\mathcal M\bigotimes_\pi \mathcal E$ is a generating
wandering subspace for the covariant representation $S\otimes
I_{\mathcal E},$ it follows from Proposition \ref{prop6} that
$(\mathcal{F}_X \bigotimes_\pi \mathcal{E}) \ominus\mathfrak
L(\mathcal{F}_X \bigotimes_\pi \mathcal{E},S\otimes I_{\mathcal
E})=\mathcal M\bigotimes_\pi\mathcal{E}$, and hence (\ref{err1})
implies that $y\in \mathcal M\bigotimes_\pi\mathcal{E}$. Hence we get
$\Pi((ker \Pi)^{\bot}\bigcap \mathcal{M}\bigotimes_\pi
\mathcal{E})=\mathcal W$.
\ep

Since each commuting tuple of operators defines a covariant
representation, the previous theorem is a generalization of
\cite[Theorem 5.2]{BEDS}. Indeed, we get the following corollary:

\begin{corollary}
With the same notation of Theorem \ref{Theorem 5.3} we have $$\bigvee_{n
\in \mathbb{Z_+}}\mathfrak L_n(\mathcal W,T)=\Pi\left(\bigvee_{n \in
\mathbb{Z_+}}\mathfrak{L}_n(F,S\otimes I_{\mathcal E})\right )$$
where  $F=(ker\Pi)^{\bot}\bigcap \mathcal{M}
\mbox{$\bigotimes_{\pi}$}\mathcal{E}$. Moreover, $F$ is wandering
subspace for the representation $S\otimes I_{\mathcal{E}}$, that is,
$F\bot \mathfrak{L}_n(F,S\otimes I_{\mathcal E})$ for each
$n\in\mathbb N.$
\end{corollary}
\bp

For each $f,f' \in F $ we have
\begin{eqnarray*}\langle f, (S_n(\zeta)\otimes I_{\mathcal{E}})f'\rangle&=&
\langle \Pi^{*}\Pi f,(S_n(\zeta)\otimes I_{\mathcal{E}})f'
\rangle=\langle \Pi f,\Pi(S_n(\zeta)\otimes I_{\mathcal{E}})f'
\rangle\\&=&0.
\end{eqnarray*}
Therefore $F$ is wandering subspace for the representation $S\otimes
I_{\mathcal E}$. Moreover, since $\mathcal{W}=\Pi F$, we have that
\[
\begin{split}
\bigvee\{\mathfrak L_n(\mathcal W,T) : n \in \mathbb{Z_+}\} & =
\bigvee\{(T_n(\zeta)\Pi(F) : \zeta \in X(n), n \in \mathbb{Z_+}\}
\\
&= \bigvee\{\Pi(S_n(\zeta)\otimes I_{\mathcal{E}})(F) : \zeta \in
X(n), n \in \mathbb{Z_+}\}
\\
&= \Pi(\bigvee \{ (S_n(\zeta)\otimes I_{ \mathcal{E}})(F) : \zeta \in X(n),  n \in \mathbb{Z_+}\})\\
&= \Pi(\bigvee\{\mathfrak{L}_n(F,S\otimes I_{\mathcal E}) : n \in
\mathbb{Z_+}\} ).\qedhere
\end{split}
\]
 \ep

\vspace{0.5cm}

\noindent{\bf Acknowledgements:} We would like to sincerely thank
the anonymous referee for valuable suggestions that improved the
content of the Section 4. The research of Sarkar was supported in
part by (1) National Board of Higher Mathematics (NBHM), India,
grant NBHM/R.P.64/2014, and (2) Mathematical Research Impact Centric
Support (MATRICS) grant, File No : MTR/2017/000522, by the Science
and Engineering Research Board (SERB), Department of Science \&
Technology (DST), Government of India. Trivedi thanks  Indian
Statistical Institute Bangalore for the visiting scientist
fellowship. Veerabathiran was supported by DST-Inspire fellowship.

\end{document}